\documentclass[11pt,a4paper]{article}
\usepackage{amsthm,amssymb,amsfonts,amsmath,lineno,mathtools,graphicx}

\usepackage{epsfig,xspace,algorithm,verbatim}
\usepackage[noend]{algpseudocode}
\usepackage{color}
\usepackage{enumerate}

\DeclarePairedDelimiter{\ceil}{\lceil}{\rceil}
\DeclarePairedDelimiter{\floor}{\lfloor}{\rfloor}

\newcommand{\N}{{\rm I\!N}}

\newcommand\blfootnote[1]{%
 \begingroup
 \renewcommand\thefootnote{}\footnote{#1}%
 \addtocounter{footnote}{-1}%
 \endgroup
}

\newcommand{\modul}[2]{ #1\equiv #2\,\, (\textrm{mod }2) }
\newcommand{\nmodul}[2]{ #1\not\equiv #2\,\, (\textrm{mod }2) }

\newtheorem{theorem}{Theorem}

\newtheorem{corollary}{Corollary}

\newtheorem{lemma}{Lemma}

\newtheorem{conjecture}{Conjecture}
\newtheorem{open}{Open problem}

\begin{document}

\title{Caterpillars are Antimagic}

\author{Antoni Lozano\thanks{Computer Science Department,
Universitat Polit\`ecnica de Catalunya, Spain, {\tt antoni@cs.upc.edu}} \and
Merc\`e Mora\thanks{Mathematics Department, Universitat Polit\`ecnica de Catalunya, Spain, {\tt
merce.mora@upc.edu}}
\and Carlos Seara\thanks{Mathematics Department, Universitat Polit\`ecnica de Catalunya, Spain, {\tt
carlos.seara@upc.edu}}
\and Joaqu\'in Tey\thanks{Math. Department, Universidad Aut\'onoma Metropolitana-Iztapalapa, M\'exico, {\tt jtey@xanum.uam.mx}}}

\date{}

\maketitle

\blfootnote{\begin{minipage}[l]{0.3\textwidth} \includegraphics[trim=10cm 6cm 10cm 5cm,clip,scale=0.15]{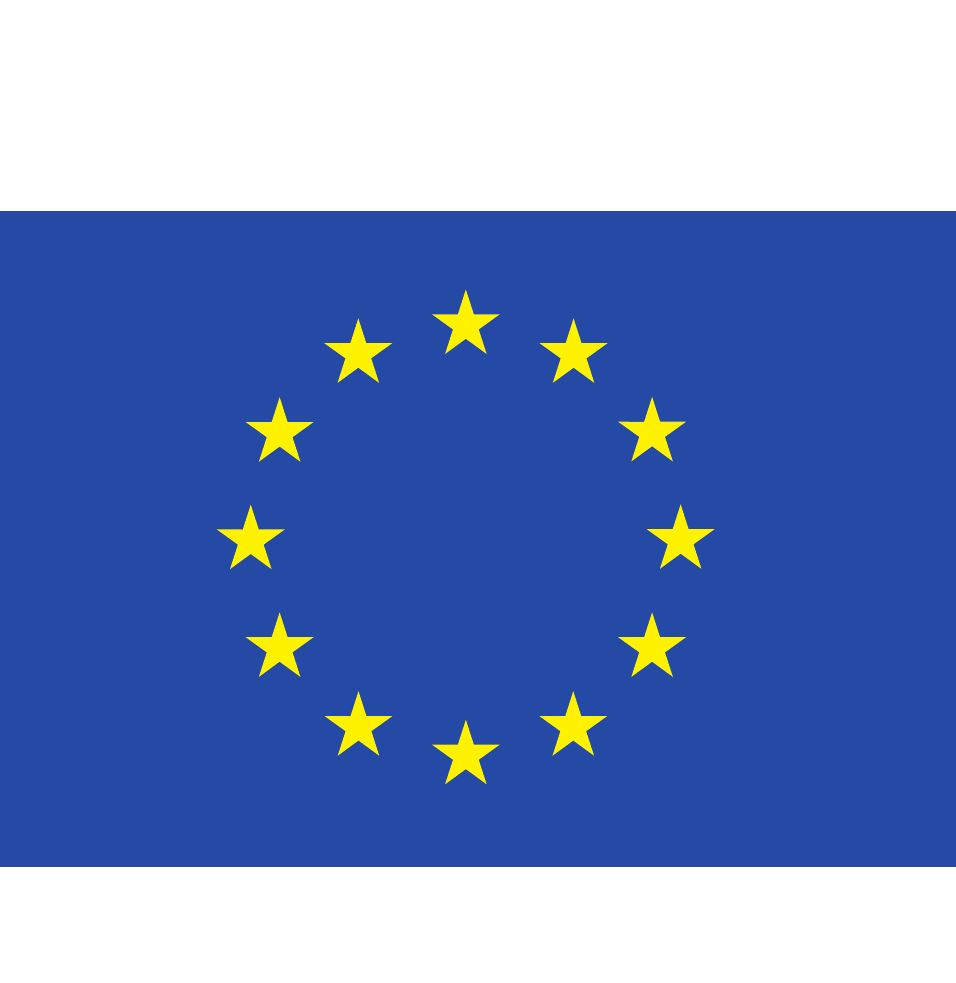} \end{minipage}  \hspace{-2cm} \begin{minipage}[l][1cm]{0.7\textwidth}
This project has received funding from the European Union's Horizon 2020 research and innovation programme under the Marie Sk\l{}odowska-Curie grant agreement No 734922.
\end{minipage}}


\begin{abstract}\noindent
An {\em antimagic labeling} of a graph $G$ is a bijection from the set of edges $E(G)$ to $\{1,2,\dots,|E(G)|\}$ such that all vertex sums are pairwise distinct, where the \emph{vertex sum} at vertex $u$ is the sum of the labels assigned to the edges incident to $u$. A graph is called \emph{antimagic} when it has an antimagic labeling. Hartsfield and Ringel conjectured that every simple connected graph other than $K_2$ is antimagic and the conjecture remains open even for trees. Here we prove that caterpillars are antimagic by means of an $O(n \log n)$ algorithm. 
\end{abstract}

\section{Introduction}\label{sec:int}

All graphs considered in this paper are finite, undirected, connected, and simple. Given a graph $G$, we denote its set of vertices by $V(G)$ and its set of edges by $E(G)$. For every vertex $v\in V(G)$, we denote by $E_G(v)$ the set of edges incident to $v$ in $G$. The degree of a vertex $v\in V(G)$ is $d_G(v)=|E_G(v)|$  (we will just write $d(v)$ when $G$ is clear from context). For undefined terminology about basic graph theory we refer the reader to~\cite{ChLZ}.

An {\em (edge) labeling} of a graph $G$ is an injection from $E(G)$ to the set of nonnegative integers. A labeling of $G$ is called \emph{antimagic} if it is a bijection $\phi:E(G)\rightarrow \{1,2,\dots,|E(G)|\}$ such that all vertex sums are pairwise distinct, where the \emph{vertex sum} at vertex $u$ is $\sum_{e \in E_G(u)} \phi(e)$. A graph is called \emph{antimagic} if it has an antimagic labeling. The following conjecture from Hartsfield and Ringel~\cite{HR} is well known.

\begin{conjecture}~\cite{HR}\label{conj1}
Every connected graph other than $K_2$ is antimagic.
\end{conjecture}

Classes of graphs which are known to be antimagic include: paths, stars, complete graphs, cycles, wheels, and complete bipartite graphs $K_{2,m}$~\cite{HR}; graphs of order $n$ with maximum degree at least $n-3$~\cite{Y}; dense graphs (i.e., graphs having minimum degree $\Omega(\log n)$) and complete partite graphs but $K_2$~\cite{AKLRY,E}; toroidal grid graphs \cite{W}; lattice grids and prisms~\cite{Ch07}; regular bipartite graphs~\cite{C}; odd degree regular graphs~\cite{CLZ}; even degree regular graphs~\cite{ChLPZ}; cubic graphs~\cite{LZ}; generalized pyramid graphs~\cite{AMPR}; graph products~\cite{WH}; and Cartesian product of graphs~\cite{Ch08,ZS} (see the dynamic survey \cite{G} for more details on antimagic labelings). However, the conjecture is still open for the general class of trees.

\begin{conjecture}~\cite{HR}\label{conj2}
Every tree other than $K_2$ is antimagic.
\end{conjecture}

One of the best known results for trees is due to Kaplan, Lev, and Roditty~\cite{KLR}, who proved that any tree having more than two vertices and at most one vertex of degree two is antimagic (see also~\cite{LWZ}). In this paper we focus on \emph{caterpillars}, that is, trees of order at least 3 such that the removal of their leaves produces a path. In \cite{LMS} the authors give sufficient conditions for a caterpillar to be antimagic and, recently, it has been shown that that caterpillars with maximum degree 3 are antimagic  \cite{DL}. In this paper we take a step further proving that every caterpillar is antimagic.

\begin{theorem}\label{th:caterpillars}
Caterpillars are antimagic. Furthermore, there exists an algorithm that, given a caterpillar $C$ of order $n$, produces an antimagic labeling for $C$ in time $O(n \log n)$.
\end{theorem}

Concretely, in Section~\ref{section:algorithm} we provide the above mentioned algorithm, whose correctness and time bound are shown in Section~\ref{section:proof}. Therefore, our method follows a constructive approach, in contrast with the use of the Combinatorial NullStellenSatz method \cite{A,H,LM}, which is the regular technique used in several of the references above.

\section{Construction of an Antimagic Labeling}\label{section:algorithm}

\begin{figure}[hb!]
\begin{center}
\includegraphics [width=0.9\textwidth]{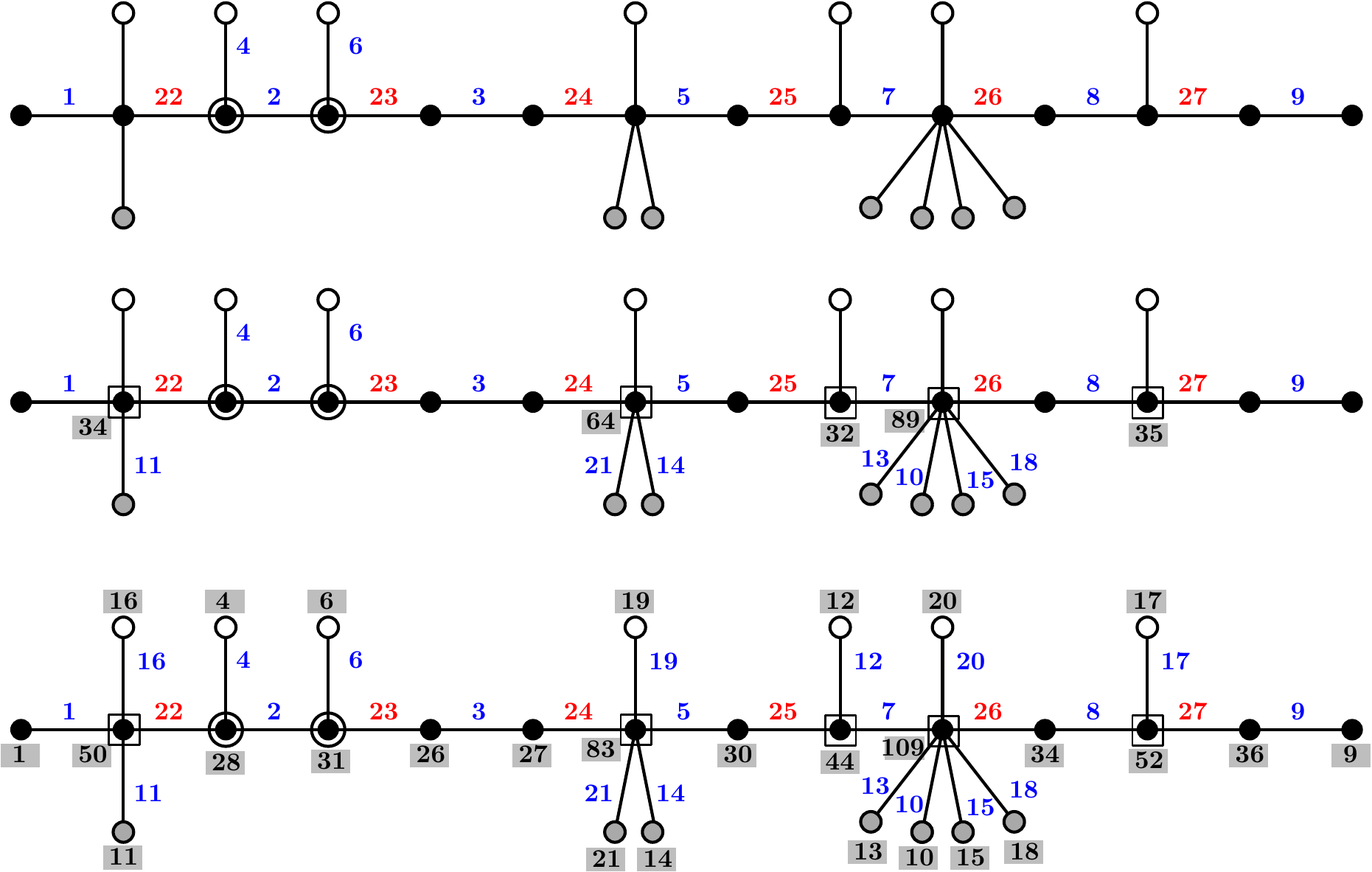}
\caption{Example of a labeling of a caterpillar with a longest path of odd length. In this case, $L_0=[1,21]$ and $L_2=[22,27]$.
Top, Step 1: Labeling the   pathedges and legs incident to light vertices (circled vertices).  
Middle, Step 2: Labeling all legs but one for each heavy vertex (squared vertices). 
Bottom, Step 3: Labeling the last leg for each heavy vertex.
}
\label{fig:exemple}
\end{center}
\end{figure}

Suppose that $C$ is a caterpillar with $m$ edges and  
let $P = (v_0,\dots,v_p)$ be a longest path in $C$.
Edges in $P$ will be called {\em   pathedges}, while edges not in $P$ will be called {\em legs}. We use the notation $e_i = \{v_{i-1},v_i\}$ ($1 \le i \le p$) for   pathedges and $f_i$ for the leg incident to a vertex $v_i$ of degree 3. 
We define $E_0=\{e_i:\modul ip \}$ and $E_1=\{e_i: \nmodul ip \}$, so that  $\{E_0,E_1\}$ is a partition of the set of  pathedges such that $e_p\in E_0$ and any two incident edges of $P$ belong to different sets. The size of these sets is $|E_0|=\ceil{p/2}$ and $|E_1|=\floor{p/2}$. 

Let $[a,b] = \{ n\in \mathbb{N}  : a\le n \le b\}$ and $[a,\infty)=\{n\in \mathbb{N} : n\ge a\}$, where $a$ and $b$ are positive integers. Split the set of available labels $L = [1,m]$ into the two subsets $L_0 = [1,m-\lfloor p/2 \rfloor]$ and $L_1 = [m-\lfloor p/2 \rfloor+1,m]$ of sizes $m-\floor{p/2} $ and $\floor{p/2} $, respectively.

\medskip
We now describe an algorithm in three steps to construct a labeling $\phi$ of $C$ that will be shown to be antimagic in Section~\ref{section:proof} (see the pseudocode in Page~\pageref{algo} and an example in Figure~\ref{fig:exemple}). 
Given a vertex $u$ in $C$ and a labeling $\phi$ of $C$, we denote the vertex sum at $u$ by $s(u)$, that is, $s(u) = \sum_{e \in E_C(u)} \phi(e)$. When we describe the construction of the labeling, we 
also use the notation $s(u)$ to refer to a {\em partial sum} at $u$, that is, the sum of the values $\phi(e)$ for all edges $e$ incident to $u$ for which a label has already been assigned at that step of the algorithm.
Similarly, we use $\phi (E(C))$ to denote the set of labels used up to that step.

\medskip
In the first step, a label is assigned to all pathedges and a few legs incident to vertices of degree 3. 
Roughly speaking, the algorithm alternatively assigns consecutive labels from the lists $L_0$ and $L_1$ to the  pathedges with some exceptions, and  taking into account that the label of edge $e_p$ must belong to $L_0$. Concretely, the edges of $E_1$ receive labels from $L_1$ in increasing order, and the edges of $E_0$ receive labels from $L_0$ in increasing order with the following exception. When assigning a label to the edge $e_i\in E_0$, $1<i\le p$, we check if the following condition holds:
\begin{description}\label{condition}
	\item[\boldmath$Q(i)$:]
	$\phi(e_{i-1})=\phi(e_k)+\phi(e_{k+1})$ for some $k\in \{1,\dots ,i-3\}$ such that $d(v_k)=3$.
\end{description}
If $Q(i)$ does not hold, then we assign the next unused label of $L_0$ to $e_i$. Otherwise, we assign the next unused label of $L_0$ to the leg $f_k$ and the next one, to $e_i$. Such a vertex $v_k$ of degree 3, whose leg $f_k$ has been labeled at this step, will be called {\em light vertex}, while the rest of vertices of degree at least three will be called {\em heavy vertices}. The set of all light vertices obtained when finishing this step is denoted by $U_p$, while the set of heavy vertices is denoted by $W$.  

In the second step, for each heavy vertex $v$, we randomly assign unused labels from $L_0$ to all but one of the legs incident to $v$.

Finally, it only remains to label one leg for each heavy vertex. To do that, in the third step we list the remaining legs in increasing order of the partial vertex sums at their corresponding incident heavy vertices. Then, we sort the remaining labels in $L_0$ in increasing order and assign them to the previous list of legs in the same order.

\begin{algorithm}[ht!]\label{algo}
\caption{Antimagic Labeling of a Caterpillar.}
\ \\ \
\hglue 5pt \textbf{Input:} A caterpillar $C$ of order $n$\\
\hglue 5pt \textbf{Output:} An antimagic labeling $\phi$ for $C$
\begin{algorithmic}[1]
    \Statex
    \Statex \hglue -4.7mm {\sc{\bf STEP 1:} Labeling the   pathedges and a few legs}
    \State $\phi(e_1) \leftarrow 1$     
    \State $\phi(e_2) \leftarrow m-\lfloor p/2 \rfloor+1$
    \State $U_1, U_2 \leftarrow \emptyset$
    \If{$p$ is even}
        \State Exchange $\phi(e_1)$ and $\phi(e_2)$
    \EndIf
    \For{$i = 3$ to $p$}
        \State $\phi(e_i) \leftarrow \phi(e_{i-2})+1$
        \State $U_i \leftarrow U_{i-1}$
        \If{$\modul ip$ and $\phi(e_{i-1}) = \phi(e_k) + \phi(e_{k+1})$        
        \Statex \hglue 1cm for some $k\in \{ 1,\dots ,i-3\}$ s.t. $d(v_k) = 3$}
            \State $\phi(f_k) \leftarrow \phi(e_i)$
            \State $\phi(e_i) \leftarrow \phi(e_i)+1$
            \State $U_i \leftarrow U_{i} \cup \{v_k\}$
        \EndIf
    \EndFor
    \State $W \leftarrow \{ v \in V(C) \mid v \notin U_p \; \wedge \; d(v) \ge 3 \}$
    \Statex
    \Statex \hglue -6mm {\sc{\bf STEP 2:} Labeling all legs except one for each vertex in $W$}
    \ForAll{$v \in W$}
        \ForAll{legs $e$ incident to $v$ except one}
            \State $\phi(e) \leftarrow$ a random label from $L_0 \setminus \phi(E(C))$
        \EndFor
    \EndFor
    \Statex
    \Statex  \hglue -6mm {\sc{\bf STEP 3:} Labeling the last leg of each vertex in $W$}
    \State Sort the vertices in $W$ as $w_1,\dots,w_t$ s.t. $s(w_i) \le s(w_{i+1})$ for all $i < t$
    \State Sort the labels in $L_0 \setminus \phi(E(C))$ as $\ell_1, \dots, \ell_t$ in increasing order
    \For{$i=1$ to $t$}
        \State $\phi(w_i) = \ell_i$
    \EndFor
\end{algorithmic}
\end{algorithm}

\section{Proof of Theorem~\ref{th:caterpillars}}\label{section:proof}

In this section we prove that the labeling produced by Algorithm 1 for a caterpillar of order $n$ is antimagic and can be found in time $O(n \log n)$, thus proving Theorem~\ref{th:caterpillars}. 
The following lemmas will be used in the proofs of correctness and efficiency of the algorithm.

We first introduce some notation. We denote by $F_0$ the set of legs incident to light vertices. Recall that the set $U_i$ introduced in the algorithm contains light vertices, that is, vertices of degree 3 whose legs have been labeled up to iteration $i$ of the \emph{for} loop in Step 1.  Hence, $|F_0|=|U_p|$. Let $E_2$ be the set of edges not belonging to $ E_0\cup E_1\cup F_0$, that is, $\{ E_0,E_1,E_2,F_0\}$ is a partition of the edge set $E(C)$.

\begin{lemma}\label{lemma:values}
For every $i\in \{1,2,\dots ,p\}$ we have
\begin{equation}\label{eq}
	\phi(e_i) =
	\begin{cases*}
	\lceil \frac{i}{2} \rceil + |U_i|, & 
	\mbox{if $e_i\in E_0$;}\\
	m - \lfloor \frac{p}{2} \rfloor + \lceil \frac{i}{2} \rceil, & 
		\mbox{if $e_i\in E_1$.}
	\end{cases*}
	\end{equation}
Moreover,
\begin{enumerate}[(1)]
	\item 	
	$ \phi(E_0 \cup F_0 )=[1,\ceil{p/2}+|U_p| ]\subseteq L_0$;
		
	\item 	
	$\phi(E_1)=L_1=[m-\floor{p/2} +1 , m ]$;	
	
	\item 
	$\phi ( E_2)= [\ceil{p/2}+|U_p|+1,m-\floor{p/2}]\subseteq L_0$.	
\end{enumerate}
\end{lemma}

\begin{proof} 
First we check that the values $\phi(e_1)$ and $\phi(e_2)$ defined in Equation~\ref{eq} are correct. 
Indeed, if $p$ is odd, then $e_1\in E_0$, $e_2\in E_1$, and the labels assigned after the execution of lines 1--5 of the algorithm are
	$\phi (e_1)=1=\ceil{1/2}+|U_1|$ and $\phi(e_2)=m - \lfloor p/2 \rfloor + 1 =m - \lfloor p/2 \rfloor + \lceil 2/2 \rceil$.
	If $p$ is even, then $e_1\in E_1$, $e_2\in E_0$, and the labels assigned in lines 1--5 are
	$\phi (e_1)=m - \lfloor p/2 \rfloor +1=m - \lfloor p/2 \rfloor + \lceil 1/2 \rceil$ and 
	$\phi(e_2)=1=\lceil 2/2 \rceil +|U_2|$.

Now, we show that the value $\phi(e_i)$,  $3 \le i \le p$,  defined in Equation~\ref{eq} corresponds to the same value calculated in the {\em for} loop at lines 6--12. Suppose that $3 \le i \le p$ and all values given in Equation~\ref{eq} up to $\phi(e_{i-1})$ are correct. Then, either $e_i \in E_0$ or $e_i \in E_1$. Suppose first that $e_i \in E_1$. Then, lines 10--12 are not executed and $\phi(e_i) = \phi(e_{i-2})+1$. Since $e_{i-2} \in E_1$ and, by hypothesis, $\phi(e_{i-2}) = m - \lfloor p/2 \rfloor + \big\lceil \frac{i-2}{2} \big\rceil$, we have
\[ \phi(e_i) = m - \Big\lfloor \frac{p}{2} \Big\rfloor + \Big\lceil \frac{i-2}{2} \Big\rceil + 1
= m - \Big\lfloor \frac{p}{2} \Big\rfloor + \Big\lceil \frac{i-2}{2} + 1 \Big\rceil
= m - \Big\lfloor \frac{p}{2} \Big\rfloor + \Big\lceil \frac{i}{2} \Big\rceil . \]
Now suppose that $e_i \in E_0$. In this case, $e_{i-2}\in E_0$ and, by hypothesis, $\phi (e_{i-2})=\big\lceil \frac{i-2}{2} \big\rceil+|U_{i-2}|$.
Note that $e_{i-1} \in E_1$ and, then, the iteration for $i-1$ produces $U_{i-1} = U_{i-2}$.
If condition $Q(i)$ does not hold, then lines 10--12 are not executed. Thus, $|U_i|=|U_{i-1}|$ and $\phi(e_i) = \phi(e_{i-2})+1$, and we get  
\[ \phi(e_i) = \Big\lceil \frac{i-2}{2} \Big\rceil +|U_{i-2}|+ 1
=  \Big\lceil \frac{i-2}{2} + 1 \Big\rceil +|U_{i-2}|
=  \Big\lceil \frac{i}{2} \Big\rceil +|U_i|. \]
Otherwise, after the execution of lines 10--12 we have that  $\phi(e_i) = \phi (e_{i-2}) + 2$ and $|U_i| = |U_{i-1}|+1$. Since $e_{i-2} \in E_0$ and, by hypothesis, $\phi(e_{i-2}) = \big\lceil \frac{i-2}{2} \big\rceil + |U_{i-2}|$, we obtain
\[ \phi(e_i) = \Big\lceil \frac{i-2}{2} \Big\rceil + |U_{i-2}| + 2
= \Big\lceil \frac{i-2}{2} + 1 \Big\rceil + |U_{i-2}| + 1
= \Big\lceil \frac{i}{2} \Big\rceil + |U_i|. \]
In both of the above cases, the value $\phi(e_i)$ computed by the algorithm coincides with that defined in Equation~\ref{eq}.

\medskip
Now, we prove the first item.
On the one hand, the values obtained for the edges in $E_0$ range between 1 and $\lceil p/2 \rceil + |U_p|$ by Equation~\ref{eq},  and the value of a leg $f_k\in F_0$ (i.e., a leg incident to a light vertex) is less than the value of some edge belonging to $E_0$. Hence, $\phi (E_0\cup F_0)\subseteq [1,\lceil p/2 \rceil + |U_p|]$. 
Since $\phi$ is an injection and $|E_0\cup F_0|= \lceil p/2 \rceil + |U_p| $, we have $\phi(E_0 \cup F_0 )=[1,\ceil{p/2}+|U_p| ]$. On the other hand, $|E_0\cup F_0|\le m-|E_1|=m-\floor{p/2}$. Thus,  $ \phi(E_0 \cup F_0 )=[1,\ceil{p/2}+|U_p| ]\subseteq L_0$.

As for the second item, since  $|E_1|=\floor{p/2}=|L_1|$, $\phi (E_1)\subseteq L_1$ and $\phi$ is an injection, we have that $\phi(E_1)=L_1=[m-\floor{p/2} +1 , m ]$. Finally, in Steps 2 and 3, the algorithm assigns the remaining labels to the remaining edges. Therefore, $\phi ( E_2 )= [1,m]\setminus \phi (E_0\cup F_0\cup E_1) =[\ceil{p/2}+|U_p|+1,m-\floor{p/2}]\subseteq L_0$ and the third item is true.
\end{proof}

\goodbreak

\begin{lemma}\label{lemma:ineq} For any $i, j$ such that $1 \le i \le  j < p$, we have:
	\begin{enumerate}[(1)]
		\item 
		$m - \lfloor p/2 \rfloor + 2 \le \phi(e_i)+\phi(e_{i+1}) \le m + \lceil p/2 \rceil + |U_p|;$
		\item If $i<j$, then 
		$ \phi(e_i) + \phi(e_{i+1}) < \phi(e_j) + \phi(e_{j+1})$;
		\item 
		If  $v_i$ and $v_j$ are  light vertices and $i<j$, then  $\phi(f_i) < \phi(f_j)$.
	\end{enumerate} 
\end{lemma}

\begin{proof}
We begin by proving the first item. 
Since one of the edges from $\{e_i, e_{i+1} \}$ belongs to $E_0$ and the other one to $E_1$,  according to Equation~\ref{eq} in Lemma~\ref{lemma:values}, we have 
	$$\phi(e_i)+\phi(e_{i+1})=m - \Big \lfloor \frac{p}{2} \Big \rfloor +\Big \lceil \frac{i}{2} \Big\rceil +\Big\lceil \frac{i+1}{2} \Big\rceil +h, $$
	for some  $h$, $0\le h\le |U_p|$. Hence, the lower bound trivially holds. For the upper bound, note that
	$$m - \Big \lfloor \frac{p}{2} \Big\rfloor +\Big\lceil \frac{i}{2} \Big\rceil +\Big\lceil \frac{i+1}{2} \Big \rceil +h\le 
	m - \Big\lfloor \frac{p}{2} \Big\rfloor +\Big\lceil \frac{p-1}{2} \Big\rceil +\Big\lceil \frac{p}{2} \Big\rceil +|U_p|\le 
	m + \Big\lceil \frac{p}{2} \Big\rceil +|U_p|.
	$$

To prove the second item, let $g: [1,p-1] \rightarrow \N$ be the function 
\[g(i) = \phi(e_i) + \phi(e_{i+1}).\]
Now, it is enough to show that function $g$ is strictly increasing. Note that for any  $i$ such that $1 \le i \le  p-2$,  proving $g(i)<g(i+1)$ is equivalent to proving the inequality:
\begin{equation}\label{eq2}
\phi(e_i) < \phi(e_{i+2})
\end{equation}
whenever $1\le i\le p-2$.
Hence, we consider two cases according to the value of $\phi $ for  pathedges given in Equation~\ref{eq}.
If $e_i \in E_0$, then
\[ \phi(e_i) < \phi(e_{i+2}) \,\, \Leftrightarrow  \,\,
\Bigl\lceil \frac{i}{2} \Bigr\rceil + |U_i| < \Bigl\lceil \frac{i+2}{2} \Bigr\rceil + |U_{i+2}|  \,\,\Leftrightarrow  \,\, |U_i| < 1+ |U_{i+2}|,  \]
which is true. Now, if $e_i \in E_1$, then
\[ \phi(e_i) < \phi(e_{i+2})  \,\, \Leftrightarrow  \,\,
m - \Bigl\lfloor \frac{p}{2} \Bigr\rfloor + \Bigl\lceil \frac{i}{2} \Bigr\rceil < m -  \Bigl\lfloor \frac{p}{2} \Bigr\rfloor + \Bigl\lceil \frac{i+2}{2} \Bigr\rceil  \,\, \Leftrightarrow  \,\, \Bigl\lceil \frac{i}{2} \Bigr\rceil < \Bigl\lceil \frac{i+2}{2} \Bigr\rceil, \]
which is also true. 
This concludes the proof of the second item.

Let us now  prove the third item. Suppose that $1\le i<j\le p-1$.
Notice that labels $\phi(f_i)$ and $\phi(f_j)$ are assigned in the first step of the algorithm if $Q(i')$ and $Q(j')$ hold when assigning the label of some  pathedges $e_{i'}$ and $e_{j'}$ of $E_0$, with $i' > i$ and $j' > j$, and  such that $\phi(e_{i'-1}) = \phi(e_i) + \phi(e_{i+1})$ and $\phi(e_{j'-1}) = \phi(e_j) + \phi(e_{j+1})$.
On the one hand, since $i<j$,
we deduce by the preceding item that
$\phi(e_{i'-1}) < \phi(e_{j'-1})$. 
Besides, $e_{i'-1},e_{j'-1}\in E_1$, and from Inequality~\ref{eq2} 
we deduce $i'<j'$.
On the other hand, the labels assigned to the edges $f_i$ and $f_j$ are  $\phi (f_i) =\phi(e_{i'-2})+1$ and $\phi (f_j) =\phi(e_{j'-2})+1$, where $e_{i'-2},e_{j'-2}\in E_0$ and $i'-2<j'-2$. Again, from Inequality~\ref{eq2}, we conclude that $ \phi (f_i)=\phi(e_{i'-2})+1<\phi(e_{j'-2})+1=\phi (f_j)$.
\end{proof}

\subsection{Proof of Correctness}\label{subsec:correctness}

Now we show the correctness of the algorithm, that is, we show that, given a caterpillar $C$, Algorithm 1 produces an antimagic labeling of $C$. 

Since $E(C)=E_0\cup E_1\cup E_2\cup F_0$, we have by Lemma~\ref{lemma:values} that  $\phi$ is a bijection from $E(C)$ onto $[1,m]$.
To prove that $\phi$ is an antimagic labeling, it remains to show that all vertex sums are pairwise different.

Consider the partition $\{V_1,V_2,V_3\}$ of $V(C)$, where $V_1$ is the set of vertices of degree 1, $V_2$ is the set of vertices of degree 2 and light vertices, and $V_3$ is the set of heavy vertices. Let 
$s(V_i) = \{s(v) : v \in V_i\}$, for $1 \le i \le 3$. It is enough to prove that for each $i$, $1 \le i \le 3$, all sums in $s(V_i)$ are pairwise distinct and the sets $s(V_1)$, $s(V_2)$, $s(V_3)$ are pairwise disjoint.

\medskip
We begin by checking that the vertices of $V_1$ have pairwise distinct vertex sums.
The vertex sum at a vertex of degree one is the value of the label of the pendant edge incident to it.
Since all edges receive different labels, vertex sums at vertices of degree one are pairwise distinct.

Now let us establish an interval of possible values in $s(V_1)$. Notice that a pendant edge is either a leg, or the first or the last edge of the path $P$.
According to Lemma~\ref{lemma:values}, legs are always labeled with elements from $L_0$. Regarding the first and last edges of the path $P$, if $p$ is odd, their labels are from $L_0$, whereas if $p$ is even, these pathedges are labeled with the smallest values of $L_1$ and $L_0$, respectively, that is, $m-\lfloor p/2 \rfloor +1$ and $1$. Hence, $s(V_1)\subseteq  L_0\cup \{ m-\lfloor p/2 \rfloor +1 \}=[1,m-\lfloor p/2 \rfloor +1]$.

\medskip
In order to check that all vertices of $V_2$ have pairwise distinct sums, consider two distinct vertices $v_j, v_k \in V_2$. Now, we distinguish three cases. In the first case, suppose that $v_j$ and $v_k$ are vertices of degree 2. If $1\le j < k<p$, then their sums can be expressed as 
\[  s(v_j) = \phi(e_j) + \phi(e_{j+1}), \; \; \; s(v_k) = \phi(e_k) + \phi(e_{k+1}), \]
and by Lemma~\ref{lemma:ineq}, $s(v_j) < s(v_k)$.
In the second case, suppose that both $v_j$ and $v_k$ are light vertices. If $1\le j < k<p$, then, similarly to the first case, their sums are 
\[ s(v_j) = \phi(e_j) + \phi(e_{j+1}) + \phi(f_j), \; \; \; s(v_k) = \phi(e_k) + \phi(e_{k+1}) + \phi(f_k)\]
and, by Lemma~\ref{lemma:ineq} we conclude that $s(v_j) < s(v_k)$. In the third case, suppose that one of the vertices, say $v_j$, has degree 2 and the other one, that is $v_k$, is a light vertex. In such a case,  we know that $Q(i)$ holds for some $i$, with $k+2<i < p$, that is, 
\[\phi(e_k)+\phi(e_{k+1}) = \phi(e_{i-1}),
\,\, \phi(f_k) = \phi(e_{i-2})+1, 
\hbox{ and }
\phi(e_i) = \phi(e_{i-2})+2 .\]
Besides, 
\[s(v_k) = \phi(e_k)+\phi(e_{k+1})+\phi(f_k)\hbox{ and }s(v_j)= \phi(e_j)+\phi (e_{j+1} ) .\]
Now, by applying Lemma~\ref{lemma:ineq}, if $ j\le k$, then
\begin{align*}
s(v_j)&=\phi(e_j)+\phi (e_{j+1})\le \phi(e_k)+\phi(e_{k+1})\\
&< \phi(e_k)+\phi(e_{k+1})+\phi(f_k) =s(v_k);
\end{align*}
if $k<j\le  i-2$, then
\begin{align*}
 s(v_j)&=\phi(e_j)+\phi (e_{j+1})\le \phi(e_{i-2})+\phi(e_{i-1}) \\
&< \phi(e_{i-1})+\phi(e_{i-2})+1\\
&= \phi(e_k)+\phi(e_{k+1})+\phi(f_k)=s(v_k);
\end{align*}
and if $k<i-1\le j$, then
\begin{align*}
s(v_j)&= \phi (e_j)+\phi(e_{j+1})  
\ge \phi (e_{i-1})+\phi(e_{i})\\
&= \phi (e_{i-1})+ \phi(e_{i-2}) +2\\
&= \phi(e_k) + \phi(e_{k+1}) + \phi(f_k)+1>s(v_k).
\end{align*}
Hence, all vertices in $V_2$ have different vertex sums.

We now determine the interval of possible values in $s(V_2)$. 
Let $v_k\in V_2$. For the lower bound, we have $s(v_k) \ge \phi(e_k)+\phi(e_{k+1}) \ge m - \lfloor p/2 \rfloor +2$ by Lemma~\ref{lemma:ineq}. For the upper bound, again by Lemma~\ref{lemma:ineq}, the vertex sum at any vertex of degree two can be bounded by $m + \lceil p/2 \rceil + |U_p|$,  while if $v_k$ is a light vertex, then for some $i$ with $k+2<i\le p$ and $\modul ip$ we have:
\[s(v_k) = \phi(e_k)+\phi(e_{k+1})+\phi(f_k) = \phi(e_{i-1}) + \phi(f_k) = \phi(e_{i-1}) + \phi(e_{i-2})+1 < m + \Big\lceil \frac{p}{2} \Big\rceil + |U_p| +1.\]
Therefore, $s(V_2) \subseteq  [m - \lfloor p/2 \rfloor+2,m + \lceil p/2 \rceil + |U_p|]$.

\medskip
Finally, recall that $V_3$ is the set of heavy vertices. Due to the way labels are assigned in Steps 2 and 3, all vertex sums of heavy vertices will be pairwise distinct. 

Let us calculate the interval of possible values of their vertex sums.
Let $v_i$ be a heavy vertex and suppose, in the first place, that $d(v_i) = 3$. Since, by construction, the leg of $v_i$ has not been labeled in Step 1, the partial sum of $v_i$ after Step 1 must be at least $m+1$. Indeed, suppose on the contrary that $\phi(e_i)+\phi (e_{i+1})\le m $. By Lemma~\ref{lemma:ineq}, we would have  $\phi(e_i)+\phi (e_{i+1})=\ell \in [ m - \bigl\lfloor p/2 \bigr\rfloor + 2, m]\subseteq L_1$, so that after assigning the label  $\ell$ to some edge $e_{j-1}$ of $E_1$ in Step 1, condition $Q(j)$ would have held for some $j$, with $j>i$, implying that the leg incident to $v_i$ would have been labeled in Step 1, a contradiction. Hence, $\phi(e_i)+\phi (e_{i+1})\ge  m +1$.
By Lemma~\ref{lemma:values}, the labels for its leg are at least $\lceil p/2 \rceil + |U_p| + 1$. Therefore, 
\[s(v_i) \ge m + \Big\lceil \frac{p}{2} \Big\rceil + |U_p| + 2.\]
Suppose now that $d(v_i) \ge 4$. Then, since there are at least two legs incident with $v_i$, we have
\begin{align*}
	 s(v_i) &\ge \phi(e_i) + \phi(e_{i+1}) + \Big(\Big\lceil \frac{p}{2} \Big\rceil + |U_p|+1 \Big) + \Big(\Big\lceil \frac{p}{2} \Big\rceil + |U_p| + 2 \Big)\\
	 &\ge  m - \Big\lfloor \frac{p}{2} \Big\rfloor + 2 + 2 \cdot \Big\lceil \frac{p}{2} \Big\rceil + 2 |U_p| + 3 \ge m + \Big\lceil \frac{p}{2} \Big\rceil + 2 |U_p| + 5. 
\end{align*}
Hence, $s(V_3) \subseteq [m + \lceil p/2 \rceil + |U_p| + 2,\infty)$.

\medskip
Notice that $s(V_1)$, $s(V_2)$, and $s(V_3)$ have been shown to be included in pairwise disjoint intervals and, as a consequence, they must also be pairwise disjoint.

\goodbreak

\subsection{Proof of Efficiency}\label{subsec:efficiency}

Finally, we show that Algorithm 1 runs in time $O(n \log n)$. 

Assignments in Step 1 of the algorithm can be done in constant time, except for the one at line 13, which requires computing a set difference and can be done in time $O(m)$. Condition in line 9 (which is equivalent to the fact that both $e_i \in E_0$ and $Q(i)$ hold) can be checked in linear time globally due to the fact that partial sums are increasing as variable $i$ increases, as shown in Lemma~\ref{lemma:ineq}. Therefore, the cost of Step 1 is the result of the linear loop at lines 6--12 and the assignment at line 13, that is, $O(m)$. Step 2 visits at most $m$ edges and assigns a random label to each of them, but labels can be chosen increasingly from the unused labels in $L_0$, thus giving a cost $O(m)$. Step 3 requires time $O(m \log m)$ due to fact that the partial vertex sums must be sorted (line 17).

The total cost of the algorithm is, then, $O(m \log m)$, but since $C$ is a tree, $n = m+1$ and the cost can be expressed as $O(n \log n)$.

\bigskip

Theorem~\ref{th:caterpillars} now follows from the proofs of correctness and efficiency contained, respectively, in Subsections~\ref{subsec:correctness} and~\ref{subsec:efficiency}.

\section{Conclusions and Open Problems}

We consider a consequence of our main result regarding oriented graphs. Define a {\em labeling} of a directed graph $D$ with $m$ arcs as a bijection from the set of arcs of $D$ to $[1,m]$. A labeling of $D$ is said to be {\em antimagic} if all oriented vertex sums are pairwise distinct, where the {\em oriented vertex sum} of a vertex in $D$ is the sum of labels of all incoming arcs minus that of all outgoing arcs. A graph is said to have an {\em antimagic orientation} if it has an orientation which admits an antimagic labeling. Hefetz, M\"utze, and Schwartz~\cite{HMS} formulate the following conjecture.

\begin{conjecture}\label{con:hef} \cite{HMS}
Every connected graph admits an antimagic orientation.
\end{conjecture}

As the authors point out in the paper, every bipartite antimagic undirected graph $G$ admits an antimagic orientation by simply orienting all edges in the same direction between the two stable sets of $G$. Therefore, we can derive the following corollary from Theorem~\ref{th:caterpillars}, which was originally proved in~\cite{L}.

\begin{corollary} \cite{L}
Caterpillars have antimagic orientations.
\end{corollary}

Lobsters, defined as trees such that the removal of their leaves produces a caterpillar, are a natural class of trees to which one can try to apply the techniques presented in this paper. A first question in the line of the above corollary would be the following.

\begin{open}
Do lobsters have antimagic orientations? 
\end{open}

More generally, we state the following question.

\begin{open}
Are lobsters antimagic?
\end{open}

\section{Acknowledgments}
Antoni Lozano is supported by the European Research Council (ERC) under the European Union's Horizon 2020 research and innovation programme (grant agreement ERC-2014-CoG 648276 AUTAR). 
Merc\`e Mora is  supported by projects Gen. Cat. DGR 2017SGR1336, MINECO MTM2015-63791-R, and H2020-MSCA-RISE project 734922-CONNECT. 
Carlos Seara is  supported by projects Gen. Cat. DGR 2017SGR1640, MINECO MTM2015-63791-R, and H2020-MSCA-RISE project 734922-CONNECT. 
Joa\-qu\'in Tey is supported by project PRODEP-12612731.

\end{document}